\documentclass[12pt]{amsart}
\usepackage{amssymb,amsmath,mathtools}
\usepackage{xcolor}
\usepackage{graphicx}
\newtheorem{proposition}{Proposition}[section]
\newtheorem{theorem}[proposition]{Theorem}

\newtheorem{lemma}[proposition]{Lemma}
\theoremstyle{definition}
\newtheorem{definition}[proposition]{Definition}
\newtheorem{remark}[proposition]{Remark}

\numberwithin{equation}{section}

\def\soglia{\begin{figure}[htb]\begin{center}
\includegraphics[width=7cm]{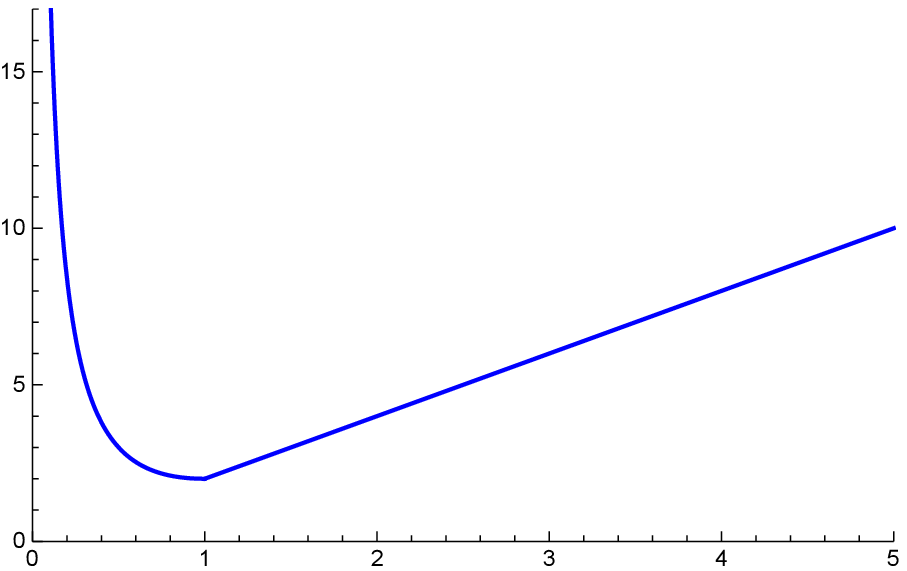}\\
{\tiny fig.\ $\!$1$\,\,$ Plot of $\t(\kappa)$ for $\alpha=\lambda_1=1$}
\end{center}
\end{figure}}

\def\soglianumerica{\begin{figure}[htb]\begin{center}
\includegraphics[width=7cm]{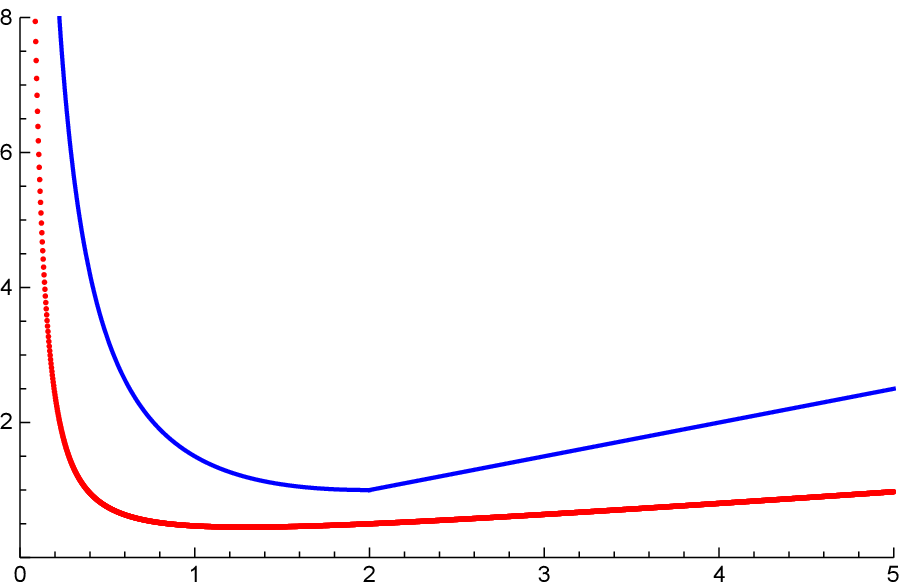}\\
{\tiny fig.\ $\!$2$\,\,$ Theoretical $\t(\kappa)$ (blue) vs actual $\t_\star (\kappa)$ computed numerically (red).}
\end{center}
\end{figure}}

\def\ratio{\begin{figure}[htb]\begin{center}
\includegraphics[width=7cm]{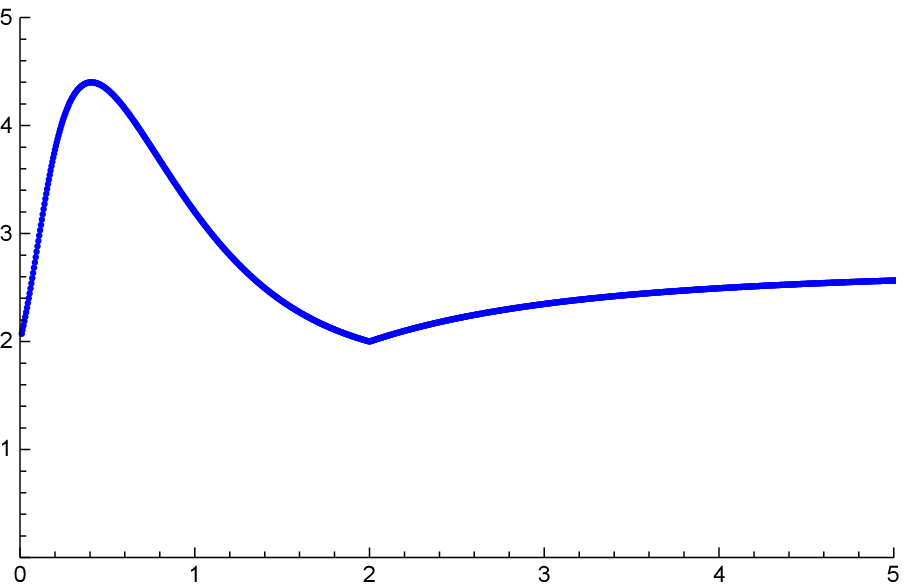}\\
{\tiny fig.\ $\!$3$\,\,$ Ratio $\t(\kappa)/\t_\star(\kappa)$ for $\kappa < 5$}
\end{center}
\end{figure}}

\def\doppio{\begin{figure}[htb]\begin{center}
\includegraphics[width=7cm]{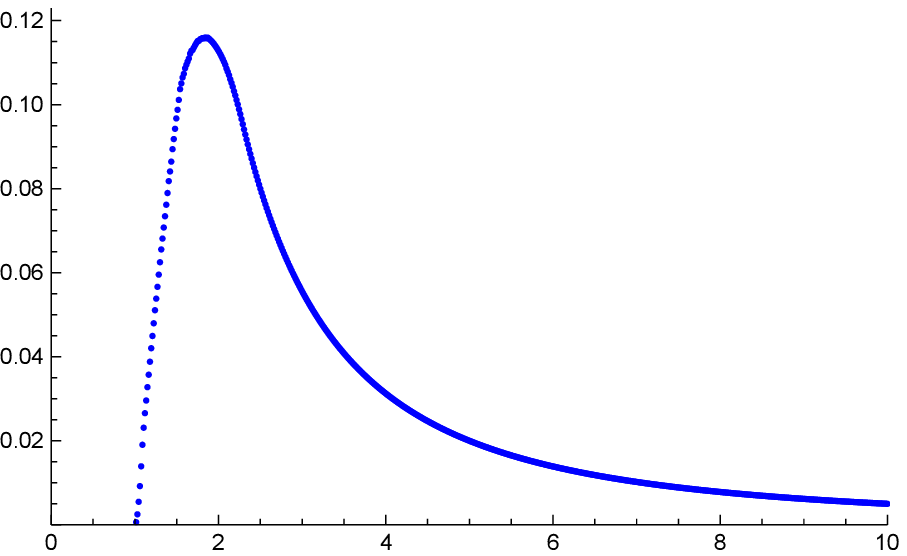}\,\,\,\includegraphics[width=7cm]{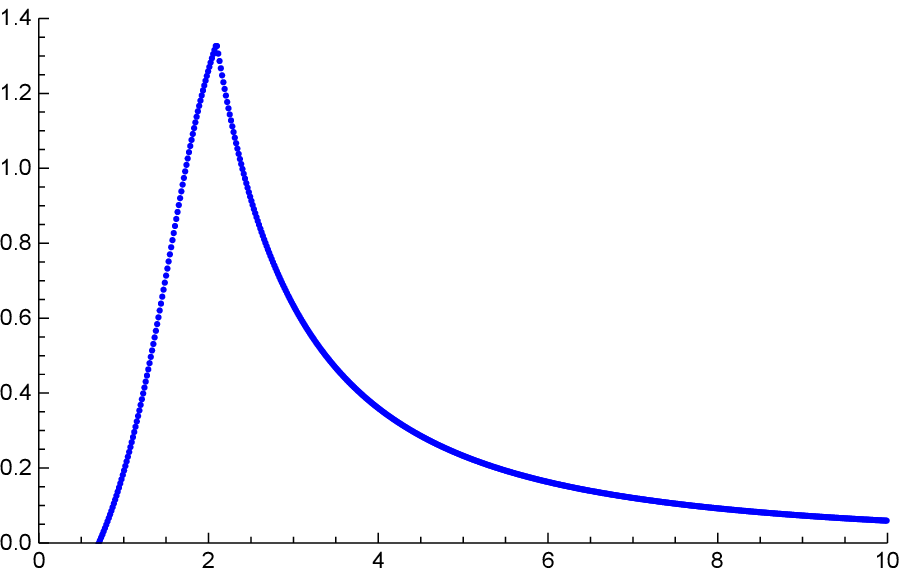}\\
{\tiny fig.\ $\!$4$\,\,$ Theoretical decay rate $\omega_b(\eta)$ (left) vs
actual decay rate $\omega_\star(\eta)$ (right) for $\eta < 10$.}
\end{center}
\end{figure}}

\def\H {{\mathcal H}}
\def\R {\mathbb{R}}

\def\t {{\mathsf t}}

\def\E {{\mathsf E}}
\def\D {{\mathfrak{D}}}

\def \l {\langle}
\def \r {\rangle}
\def \and {{\qquad\text{and}\qquad}}
\def \au {\rm}
\def \ti {\it}
\def \jou {\rm}
\def \bk {\it}
\def \no#1#2#3 {{\bf #1} (#3), #2.}
\def \eds#1#2#3 {#1, #2, #3.}
\title[Supercritical MGT-Fourier model]
{The MGT-Fourier model\\
in the supercritical case}

\author[M. Conti, L. Liverani, V. Pata]{Monica Conti, Lorenzo Liverani, Vittorino Pata}

\address{Politecnico di Milano - Dipartimento di Matematica
\newline\indent
Via Bonardi 9, 20133 Milano, Italy}
\email{monica.conti@polimi.it {\rm (M. Conti)}}
\email{lorenzo.liverani@polimi.it {\rm (L. Liverani)}}
\email{vittorino.pata@polimi.it {\rm (V. Pata)}}


\begin{document}

\begin{abstract}
We address the energy transfer in the differential system
$$
\begin{cases}
u_{ttt}+\alpha u_{tt} - \beta \Delta u_t - \gamma \Delta u = -\eta \Delta \theta \\
\theta_t - \kappa \Delta \theta =\eta \Delta u_{tt}+ \alpha\eta \Delta u_t
\end{cases}
$$
made by a Moore-Gibson-Thompson equation in the supercritical regime, hence antidissipative,
coupled with the classical heat equation.
The asymptotic properties of the related solution semigroup depend
on the strength of the coupling, ruling the competition between the Fourier damping and
the MGT antidamping.
Exponential stability will be shown always to occur, provided that the coupling constant
is sufficiently large with respect to the other structural parameters.
A fact of general interest will be also discussed, namely, 
the impossibility of attaining the optimal exponential decay rate of a given
dissipative system via energy estimates.
\end{abstract}

\maketitle

\section{Preamble: The MGT Equation}

\noindent
The Moore-Gibson-Thompson (MGT) equation
is the third-order in time PDE
\begin{equation}
\label{MGT}
\partial_{ttt} u + \alpha \partial_{tt} u-\beta \Delta \partial_t u - \gamma \Delta u =0,
\end{equation}
ruling the evolution of the unknown variable
$u=u(\boldsymbol{x},t):\Omega\times [0,\infty) \to \R$,
where $\Omega\subset\R^N$ is a bounded domain with sufficiently smooth boundary $\partial\Omega$.
Here, $-\Delta$ is the Laplace-Dirichlet operator, while
$\alpha,\beta,\gamma>0$ are fixed structural parameters.

\smallskip
Equation~\eqref{MGT} originally arises in the modeling of wave
propagation in viscous thermally relaxing
fluids~\cite{MOG,THO}. Nonetheless, its first appearance goes back to a very old
paper of Stokes~\cite{STO}.
Later on, several authors understood that the MGT equation may pop up in
the description of a large variety of physical phenomena, ranging
from viscoelasticity to thermal conduction.
In particular, we want to highlight the interpretation of
\eqref{MGT} as a model for the vibrations in a standard linear
viscoelastic solid \cite{DAC,GOB}.
Such a model can also be obtained as a particular case of the
equation of linear viscoelasticity deduced within a rheological framework
using a linear combination of springs and dashpots (see e.g.\ \cite{DRO})
$$
\partial_{tt}u(t) - g(0)\Delta u(t) - \int_{0}^\infty g'(s)\Delta u(t-s)\,d s=0,
$$
upon choosing the exponential kernel
$$g(s)=\varkappa e^{-\alpha s}+\frac{\gamma}{\alpha},
$$
where the constant
\begin{equation}
\label{stabby}
\varkappa = \beta - \frac{\gamma}{\alpha}
\end{equation}
must be strictly positive (see \cite{DPMGT}).

\smallskip
As far as the mathematical analysis of \eqref{MGT} is concerned, there is nowadays a vast literature
(see, e.g., \cite{BU1,BU2,DPMGT,KLM,KLP,KN,MDT} and references therein).
Let us briefly subsume the main results obtained so far.
For every choice of the parameters
$\alpha,\beta,\gamma>0$, the MGT equation turns out
to generate a strongly continuous semigroup of solutions
on the natural weak energy space
$$
H^1_0(\Omega) \times H^1_0(\Omega) \times L^2(\Omega).
$$
Here, with standard notation, $L^2$ is the Lebesgue space of square summable functions,
while $H_0^1$ is the Sobolev space of square summable functions along with their first derivatives,
with null trace on the boundary $\partial\Omega$.
However, the asymptotic behavior of the solutions dramatically depends on the
constant $\varkappa$ defined in \eqref{stabby}, which in the MGT context is usually referred to as the
\emph{stability number}.
Indeed, depending on the sign of $\varkappa$, the equation may exhibit
dissipative or antidissipative features.
More precisely, we have the following
picture:
\begin{itemize}
\item[$\diamond$] If $\varkappa>0$ the solutions decay exponentially fast.
\smallskip
\item[$\diamond$] If $\varkappa=0$ the (nontrivial) solutions are bounded but do not decay.
\smallskip
\item[$\diamond$] If $\varkappa<0$ there are solutions with an exponential blow up.
\end{itemize}
For this reason, the regimes $\varkappa>0$, $\varkappa=0$ and $\varkappa<0$ are usually referred to
in the literature as
\emph{subcritical}, \emph{critical} and \emph{supercritical}, respectively.
In particular, when $\varkappa=0$
there is the conservation of an appropriate energy, equivalent to the square norm in the phase space.
In view of our previous discussion, we may conclude that the MGT equation is
a model of viscoelasticity \emph{only} in the subcritical regime,
whereas when $\varkappa\leq 0$ the viscoelastic interpretation
is completely lost. Nonetheless the critical and supercritical regimes
remain very interesting, both from the physical and the mathematical viewpoint.
In the critical regime, one might also ask whether or not a further
damping mechanism (e.g., of memory type) could induce the uniform decay of the energy,
hence driving the regime into subcritical. In general, this is false: an
example in this direction has been given in~\cite{DLP}.
In fact, each damping mechanism is peculiar, and it may require a sharp dedicated analysis.
It is then not a coincidence that almost all
the results available in the literature are set in the subcritical regime.

\section{Introduction}

\subsection{The MGT-Fourier model}
The object of our analysis is the system obtained by coupling the MGT equation~\eqref{MGT}
with the classical Fourier heat equation
\begin{equation}
\label{main_sys}
\begin{cases}
u_{ttt}+\alpha u_{tt} - \beta \Delta u_t - \gamma \Delta u =- \eta \Delta \theta, \\
\noalign{\vskip1mm}
\theta_t - \kappa \Delta \theta =  \eta \Delta u_{tt} + \alpha\eta \Delta u_t,
\end{cases}
\end{equation}
supplemented with the Dirichlet boundary conditions
$$
u(\boldsymbol{x},t) = \theta(\boldsymbol{x},t) = 0,\quad \boldsymbol{x} \in \partial \Omega.
$$
Here $\alpha, \beta, \gamma > 0$ are the usual MGT parameters,
$\kappa > 0$ is the thermal conductivity, while $\eta \neq 0$ is the coupling constant.

\smallskip
System~\eqref{main_sys} has been first addressed by the authors of~\cite{BUR},
within the key assumption
that the stability number $\varkappa$ defined in~\eqref{stabby}
be strictly positive. In which case,
there is a clear physical interpretation, namely, a thermoviscoelastic model describing the vibrations
in a viscoelastic heat conductor obeying the Fourier
heat conduction law. The main result of~\cite{BUR}, besides the generation
of the solution semigroup, is that
the total energy decays exponentially fast
for every value of the coupling constant $\eta$.
Let aside the undeniable interest for the model, which also happens to be the first example
in the literature of a coupled MGT equation,
the predicted exponential
decay is not that surprising. Indeed, each single equation
of \eqref{stabby} gives rise to an exponentially stable semigroup. In addition,
the coupling is fairly nice, in the sense that when one performs the basic estimates
in the weak energy space, the contributions produced by the $\eta$-terms
(which in principle could create problems) cancel each other.
In fact, the action of the coupling is typically the one of transferring energy,
as well as dissipation, between the equations of a given system. But in this case both equations
exhibit enough dissipation by themselves from the very beginning. It is also worth mentioning that
for the analogous system modeling a thermoviscoelastic plate
$$
\begin{cases}
u_{ttt}+\alpha u_{tt} + \beta \Delta^2 u_t + \gamma \Delta^2 u =- \eta \Delta \theta, \\
\noalign{\vskip1mm}
\theta_t - \kappa \Delta \theta =  \eta \Delta u_{tt} + \alpha\eta \Delta u_t,
\end{cases}
$$
where in the MGT equation $-\Delta$ is replaced by the bilaplacian,
the solution semigroup is not only exponentially stable,
but also analytic (see~\cite{CPPQ}).

\subsection{A general discussion}
There are instead interesting situations where a system is made, say, by two equations,
one of which dissipates through a damping mechanism, while the other one preserves the energy.
Now the coupling becomes essential, since its action allows to transfer dissipation to the
undamped equation, in such a way that the whole system becomes globally stable
as time goes to infinity.
Several examples of this kind can be found in the literature. For instance, we address the reader to the works
\cite{R1,R2,R3,R4,R5,CPPQ2,R6,R7,R8,R9,R10,R11,R12,R13,R14,R15,R16,R17,R18,R19}, just to name a few.
In some of these contributions, the dissipation is not mechanical,
but only thermal through a heat equation
of Fourier type. On the other hand, it is well known that the latter possesses highly regularizing properties,
due to the fact that the Fourier damping mechanism is particularly strong.
As a byproduct, even a very small coupling (e.g., with a coupling constant $|\eta|\ll 1$)
is enough to get stability.
But let now push our discussion a little bit further.
We may ask what happens if we couple a \emph{dissipative}
equation with an \emph{antidissipative} one.
The picture now is more intriguing, and the strength of the coupling comes into play.
To hope for stability, it is necessary that the equations
share their energies to a certain extent. This translates into the fact that the coupling cannot be
too weak. But even if the coupling is strong enough to bypass a certain critical threshold,
the system may remain unstable if the action of the antidamping is more effective than the one
of the damping. To the best of our knowledge, the first analysis of this kind
has been made in the very recent paper~\cite{CLP},
where a simple (yet not so simple) system of ODEs is considered, that is,
$$
\begin{cases}
\ddot u+u+a\dot u=\eta\dot v,\\
\ddot v+v-b \dot v=-\eta\dot u,
\end{cases}
$$
$a,b>0$ being a damping and an antidamping parameter, respectively.
Here, stability occurs only if $a>b$ and $|\eta|$ is sufficiently large.
Loosely speaking, we need not only more damping than antidamping, but also
a fairly good communication between the two equations.
There is also another interesting issue in connection with the case $a>b$, namely,
to find the value
of $\eta$ ensuring the best decay. The answer, in contrast to what one might think, is not $|\eta|$ as large
as possible; on the contrary, there is an optimal finite value of $|\eta|$, depending on
the parameters $a$ and $b$. However, this is not the
general rule.
For instance, one can construct a similar system of two oscillators, but with a different coupling involving
$u$ and $v$ rather than their derivatives, where the best decay rate is reached asymptotically when $|\eta|\to\infty$.

\subsection{The supercritical case}
In the light of our previous comments, we can now move to the core of the present paper,
which can be summarized by the following question:
\smallskip
\begin{center}
\begin{minipage}{14cm}
\emph{What happens if we consider the MGT-Fourier system \eqref{main_sys}
when the stability number $\varkappa$ is zero, or even negative?}
\end{minipage}
\end{center}
\smallskip
If $\varkappa\leq 0$, the MGT equation is no longer dissipative, and the
only dissipation mechanism is contributed by the heat equation. Even more so, if $\varkappa$ is strictly negative,
then we have a competition between the two equations, as the MGT one
becomes antidissipative.
In this case, we will prove that if the coupling constant is very small in modulus, the two equations
are almost unrelated, and the explosive character of the MGT one becomes predominant,
pushing the total energy to exponential blow up.
Instead, if $|\eta|$ is sufficiently large and $\varkappa$ slightly negative,
it is reasonable to expect that
the dissipation provided by the Fourier component should prevail. But as $\varkappa\to -\infty$ the
antidamping becomes stronger and stronger. Thus, in principle, it is hard to predict if stability can
still be attained via
the coupling. Nevertheless, for \emph{any} value of $\varkappa$, no matter how negative,
there exists a critical threshold for $|\eta|$ beyond which exponential stability occurs, meaning that
the damping mechanism of the classical heat equation is stronger than any possible structural antidamping
of the MGT equation. From the mathematical viewpoint, the main difficulty lies in the fact
that in \cite{BUR} all the estimates are obtained using the functional
$$
\mathsf W  = \int_\Omega |u_{tt}+ \alpha u_t|^2 d\boldsymbol{x}
+ \frac{\gamma}{\alpha}\int_\Omega |\nabla u_t+ \alpha \nabla u|^2 d\boldsymbol{x}
+ \varkappa\int_\Omega |\nabla u_t|^2 d\boldsymbol{x}
+ \int_\Omega |\theta|^2 d\boldsymbol{x},
$$
which, for $\varkappa>0$, turns out to be equivalent to the energy, and appears in a natural way when performing the basic
multiplications dictated by the phase space of the problem.
Unfortunately, $\mathsf W$ becomes a pseudoenergy when $\varkappa=0$,
so that its decay provides no information on the decay of the energy itself.
And $\mathsf W$ ceases to be even a pseudoenergy when
$\varkappa<0$, as it may assume negative values.
Accordingly, a more refined analysis is required, via further energy-like functionals.
Each of them contributes with terms with the right sign, but introducing
at the same time terms playing against dissipation. A delicate balance is needed, in order to
convey the whole system towards stability.

\subsection{The results}
We consider system~\eqref{main_sys}, where the MGT equation
lies either in the critical or in the supercritical regime. Hence, defining
$$
\mu = \gamma - \alpha \beta = - \alpha \varkappa,
$$
we restrict our analysis to the novel case
$\mu\geq 0$.
After proving the existence of the solution semigroup $S(t)$ on the natural weak energy space
$$
\H = H^1_0(\Omega) \times H^1_0(\Omega) \times L^2(\Omega)\times L^2(\Omega),
$$
we will show that, for every value $\mu \geq 0$,
the exponential decay of the energy takes place provided that the coupling constant
$\eta$ is sufficiently large in modulus.
More precisely, there exists a structural threshold $\t>0$, independent of $\eta$,
such that exponential stability occurs whenever
$$\eta^2 > \t\mu.$$
This tells in particular that in the critical regime $\mu=0$ we have
exponential stability for \emph{every} $\eta \neq 0$. Conversely, in the supercritical regime $\mu>0$,
exponentially growing trajectories always arise if $|\eta|$ is small.

Aiming to deepen our comprehension of the phenomenon in the supercritical regime $\mu>0$,
a further question is to be addressed:

\medskip
\centerline{\emph{How does the thermal conductivity $\kappa$ influence the stability threshold $\t$?}}
\medskip

\noindent
We will see that $\t\to \infty$ when $\kappa\to 0$, which is highly expected: if the heat equation is weakly
damped, a large coupling constant is needed in order to transfer enough dissipation
to the mechanical part. Much more surprising, at first glance, is that $\t\to \infty$
when $\kappa\to\infty$ either, as one might reckon that a stronger dissipation in the Fourier law would
transmit a stronger dissipation to the whole system. If so, one should obtain
the uniform decay of the energy by keeping $\eta$ fixed, upon arbitrarily increasing $\kappa$.
On the contrary, we propose the following physical interpretation, complying with our findings:
If $\kappa$ is very large, then $\theta$ tends to decay in a very short time, and the effect is that
the mechanical part, which would blow up in absence of heat interaction, does not see the coupling, unless
$|\eta|$ is sufficiently large compared with $\kappa$.

A further issue concerns with the decay rate in dependence of $\eta$, once all the other quantities are fixed.
What happens is that the optimal decay occurs for a certain value
$\pm \eta_\star $ of the coupling constant, depending on the
structural parameters, that clearly
satisfies the condition $\eta_\star ^2>\t\mu$, but at the same time is relatively small in modulus.
So there exists the most efficient coupling in terms of energy transmission.
And when $|\eta|\to\infty$, the decay rate eventually deteriorates, becoming zero in the limit.
But finding the \emph{exact} value of the best decay rate is quite another story.
Indeed, in the final Appendix we will show that, in general and in particular in our case,
establishing the best exponential decay rate
via energy estimates is a hopeless task.

\section{Functional Setting and Notation}
\label{SFSN}

\noindent
Let $(H,\l\cdot,\cdot\r,\|\cdot\|)$ be a real Hilbert space,
and let $A:H\to H$ be a strictly positive selfadjoint operator
with domain $\D(A)\subset H$, where the embedding is not necessarily compact.
For $\sigma \in\R$, we define the hierarchy of continuously nested Hilbert spaces
$$H^\sigma = \D(A^{\sigma/2}),$$
endowed with the scalar products and norms ($\sigma$ will be always omitted whenever zero)
$$\l u,v\r_{\sigma}
=\l A^{\sigma/2}u,A^{\sigma/2}v\r\qquad\text{and}\qquad
\|u\|_\sigma=\|A^{\sigma/2}u\|.$$
For $\sigma>0$, it is understood that $H^{-\sigma}$
denotes the completion of the domain, so that $H^{-\sigma}$ is the dual space of $H^\sigma$.
The symbol $\l\cdot,\cdot\r$ will also denote the duality pairing between
$H^{-\sigma}$ and $H^\sigma$. We recall the Poincar\'e inequality
$$\|u\|_{\sigma-1}\leq \frac1{\sqrt{\lambda_1}\,}\|u\|_\sigma, \quad\forall u\in H^\sigma,$$
where $\lambda_1>0$ is the minimum of the spectrum of $A$ (the first eigenvalue if the spectrum is discrete).
Finally, we introduce the phase space of our problem, namely, the product Hilbert space
$$\H = H^1 \times H^1 \times H \times H,$$
endowed with the norm
$$\|(u,v,w,\theta)\|_{\H}^2=\|v + \alpha u\|^2_1+\|w+ \alpha v\|^2
+ \|v\|^2_1 + \|\theta\|^2.$$
Such a norm, which comes from the MGT structure, is well known to be equivalent
to the standard product norm in $\H$ (see, e.g., \cite{DPMGT}).

\section{The Solution Semigroup}
\label{TSS}

\noindent
In greater generality, we consider the abstract problem
\begin{equation}
\label{main_sysA}
\begin{cases}
u_{ttt} + \alpha u_{tt} + \beta A u_t + \gamma A u =\eta A \theta, \\
\noalign{\vskip1mm}
\theta_t + \kappa A \theta =- \eta A u_{tt} - \alpha \eta A u_t,
\end{cases}
\end{equation}
in the MGT critical or supercritical regime, i.e., with
$$\mu = \gamma - \alpha \beta\geq 0.$$
The system is subject to the initial conditions
\begin{equation}
\label{main_sysA-IC}
\begin{cases}
u(0) = u_0, \\
u_t(0) = v_0, \\
u_{tt}(0) = w_0, \\
\theta(0) = \theta_0,
\end{cases}
\end{equation}
where $\boldsymbol{u}_0=(u_0,v_0,w_0,\theta_0)\in\H$ is an arbitrarily given initial datum.

\begin{remark}
For the concrete system~\eqref{main_sys} of the Introduction, $A$ is the
Laplace-Dirichlet operator $-\Delta$ acting on the Hilbert space
$H=L^2(\Omega)$, while $H^1=H_0^1(\Omega)$.
\end{remark}

The first result addresses the well-posedness of the problem.

\begin{theorem}
\label{existence}
For every $\boldsymbol{u}_0\in\H$, problem~\eqref{main_sysA}-\eqref{main_sysA-IC}
admits a unique weak solution
$$t\mapsto \boldsymbol{u}(t)=(u(t),u_t(t),u_{tt}(t),\theta(t))\in{\mathcal C}([0,\infty),\H),$$
satisfying, for every fixed $t\geq0$, the further continuity property
$$\boldsymbol{u}_0\mapsto \boldsymbol{u}(t)\in {\mathcal C}(\H,\H).$$
\end{theorem}

Accordingly, system~\eqref{main_sysA} generates a strongly continuous semigroup~\cite{CZW,PAZ}
$$S(t):\H\to \H,$$
acting by the rule
$$S(t)\boldsymbol{u}_0=\boldsymbol{u}(t),$$
whose corresponding energy at time $t$, for the initial datum $\boldsymbol{u}_0$, reads
$$\E(t)=\frac12\|S(t)\boldsymbol{u}_0\|^2_\H=\frac12\Big[\|u_t(t) + \alpha u(t)\|_1^2+\|u_{tt}(t) + \alpha u_t(t)\|^2
+\|u_t(t)\|_1^2+\|\theta(t)\|^2\Big].$$

\begin{proof}[Proof of Theorem \ref{existence}]
The conclusion follows by showing that the energy $\E(t)$
of any Galerkin approximate solution is
uniformly bounded on every time-interval $[0,T]$, with a bound 
of the form $\E(0)h(T)$, for some positive increasing function $h$.
Indeed, from the one side this 
produces the required uniform bound of the solution. From the other side,
since the equation is linear, the same bound holds
for the difference of two Galerkin approximants, yielding the
convergence of the entire approximating sequence to its (unique) limit 
in the topology of $\mathcal{C}([0,T],\H)$.
By the same token, the energy of the difference of two solutions
satisfies the same estimate, yielding the continuous dependence.

To this end, we set
$$\alpha_m=\alpha+m,$$
where we choose $m>0$ large enough that
$$
\varkappa_m = \beta - \frac{\gamma}{\alpha_m} =\frac{m\beta-\mu}{\alpha+m}> 0.
$$
This trick, first devised in \cite{DPMGT}, allows to rewrite \eqref{main_sysA} in the form
$$
\begin{cases}
u_{ttt} + \alpha_m u_{tt} + \beta A u_t + \gamma A u=\eta A \theta + m u_{tt}, \\
\noalign{\vskip1mm}
\theta_t + \kappa A \theta =- \eta A u_{tt} -\alpha_m \eta A u_t + m \eta A u_t.
\end{cases}
$$
Note that now the first equation is
a subcritical MGT one, plus some lower order terms,
which can be multiplied by the standard MGT multiplier $u_{tt} + \alpha_m u_t$, to get
\begin{align*}
&\frac{d}{dt}\Big[\frac{\gamma}{\alpha_m}\|u_t + \alpha_m u\|_1^2 +\|u_{tt} + \alpha_m u_t\|^2
+ \varkappa_m \|u_t\|^2_1 \Big] + 2\alpha_m \varkappa_m \|u_t\|^2_1 \\
& =  2\eta \l A \theta, u_{tt} +\alpha_m u_t\r + 2m \l u_{tt}, u_{tt}+\alpha_m u_t \r.
\end{align*}
Multiplying instead the second equation of the new system by $\theta$, we obtain
$$
\frac{d}{dt}\|\theta\|^2+ 2\kappa \|\theta\|_1^2= - 2\eta \langle A u_{tt}
+ \alpha_m A u_t, \theta\rangle +2 m \eta \l Au_t, \theta \r.
$$
Defining the energy functional
$$
{\mathsf W}_m(t) =
\frac{\gamma}{\alpha_m}\|u_t(t) + \alpha_m u(t)\|^2_1+\|u_{tt}(t)+ \alpha_m u_t(t)\|^2
+ \varkappa_m\|u_t(t)\|^2_1 + \|\theta(t)\|^2,
$$
and adding the two differential identities, we end up with
\begin{equation}
\label{meddler}
\frac{d}{dt}{\mathsf W}_m + 2\alpha_m\varkappa_m \|u_t\|^2_1 + 2\kappa \|\theta\|^2_1
=  2m \l u_{tt}, u_{tt}+\alpha_m u_t \r +2m \eta \l Au_t, \theta \r.
\end{equation}
Note that ${\mathsf W}_m$ is equivalent to the original energy $\E$ (and we write  ${\mathsf W}_m\sim \E$),  that is,
$$\frac1c\, \E\leq {\mathsf W}_m\leq c\,\E,$$
for some $c>1$. This is true because of the strict
positivity of the parameter $\varkappa_m$. Making use of the
Young and Poincar\'e inequalities, along with the equivalence
above, the right-hand side of \eqref{meddler} is immediately controlled by
$$2\alpha_m\varkappa_m \|u_t\|^2_1 + 2\kappa \|\theta\|^2_1 +k {\mathsf W}_m,$$
for some $k>0$. This leads to the differential inequality
$$\frac{d}{dt}{\mathsf W}_m \leq k {\mathsf W}_m,$$
and a final application of the Gronwall lemma yields
$${\mathsf W}_m(t)\leq {\mathsf W}_m(0)e^{kt},$$
providing the sought uniform bound.
\end{proof}

\begin{remark}
Clearly, the proof remains valid in the subcritical regime $\mu<0$, where one can simply take $m=0$.
In which case, \eqref{meddler} actually tells that $S(t)$ is a contraction semigroup with respect to the equivalent
norm in $\H$ dictated by ${\mathsf W}_0$.
\end{remark}

Once the well-posedness result is established,
the \emph{quasienergy}
$$
\mathsf W(t) =\frac{\gamma}{\alpha}\|u_t(t) + \alpha u(t)\|_1^2+\|u_{tt}(t) + \alpha u_t(t)\|^2
- \frac{\mu}{\alpha}\|u_t(t)\|_1^2+\|\theta(t)\|^2,
$$
corresponding to a solution $\boldsymbol{u}(t)$,
is easily seen to fulfill the equality
\begin{equation}
\label{EnEq}
\frac{d}{dt} \mathsf W(t) + 2\kappa\|\theta(t)\|^2_1 = 2\mu \|u_t(t)\|^2_1,
\end{equation}
for all sufficiently regular initial data.
Indeed, $\mathsf W$ is nothing but the functional $\mathsf W_0$ of the previous proof,
hence \eqref{EnEq} is merely obtained by setting $m=0$ in \eqref{meddler}.
In particular, all the terms containing the coupling constant $\eta$ disappear.
The reason why we refer to ${\mathsf W}$ as a quasienergy is that it may assume negative values
when $\mu>0$, whereas it is actually a \emph{pseudoenergy} in the critical regime $\mu=0$.
Although ${\mathsf W}$, contrary to what happens in the subcritical case analyzed in \cite{BUR}, cannot possibly be
equivalent to the energy, the identity~\eqref{EnEq} will turn out to be crucial in order to prove
the exponential stability of the semigroup.

\section{Exponential Blow Up}
\label{EBU}

\noindent
In the MGT supercritical regime $\mu>0$, we show that
system \eqref{main_sysA} exhibits solutions whose energy blows up exponentially fast,
whenever the coupling constant
$\eta$ is small in modulus.
We suppose for simplicity that the operator $A$ has at least one eigenvalue $\lambda>0$.
Choosing then a corresponding eigenvector $w$, we look for solutions to \eqref{main_sysA} of the
form
$$u(t)=\phi(t)w\qquad\text{and}\qquad\theta(t)=\psi(t)w,$$
for some $\phi\in {\mathcal C}^3([0,\infty))$
and $\psi\in {\mathcal C}^1([0,\infty))$. The functions $\phi$ and $\psi$ are easily seen to fulfill the
linear system of ODEs
\begin{equation}
\label{ODE}
\begin{cases}
\phi''' + \alpha \phi'' + \beta \lambda \phi' + \gamma \lambda \phi -\eta \lambda \psi=0, \\
\noalign{\vskip1mm}
\psi' + \kappa \lambda \psi + \eta \lambda \phi'' + \alpha \eta \lambda \phi'= 0.
\end{cases}
\end{equation} From the classical ODE theory~\cite{HS}, it is well known that
the asymptotic properties of the solutions to~\eqref{ODE} are completely determined
by the (complex) roots
of the characteristic polynomial
$$
p_\eta(z)= z^4 + (\alpha + \kappa \lambda)z^3 + (\beta \lambda + \alpha \kappa \lambda +
 \eta^2 \lambda^2) z^2 + (\gamma \lambda+ \beta \kappa \lambda^2 + \alpha \eta^2 \lambda^2 ) z + \gamma \kappa \lambda^2.
$$
In particular, if $p_\eta$ has a root with strictly positive real part, then there are
exponentially blowing up solutions. Accordingly, our claim follows from
the next proposition.

\begin{proposition}
\label{PROPOPRIMA}
Let $\mu > 0$. Then, for every $\eta\in\R$ with $|\eta|$ small enough,
$p_\eta$ possesses a complex root $z_\eta$ with $\Re z_\eta>0$.
\end{proposition}

\begin{proof}
We first examine the case $\eta=0$, corresponding to the uncoupled system,
whose characteristic polynomial simplifies into
$$p_0(z)=(z+\kappa\lambda)q_0(z),$$
where
$$
q_0(z)= z^3 + \alpha z^2+ \beta \lambda z +\gamma \lambda.
$$
The roots of $q_0$ have been carefully analyzed in \cite{DPMGT},
where it is proved that when $\mu>0$ there are
always two complex conjugate solutions with positive real part,
for every fixed $\lambda>0$.
Therefore, $p_0$ has a root $z_0$ with $\Re z_0>0$.
At this point, we use the continuous dependence of the roots of a polynomial on its
coefficients (see \cite{US}).
Indeed, it is apparent that the coefficients of the polynomial $p_\eta$ converge to those of
$p_0$ as $\eta\to 0$. This implies that, for every $\varepsilon>0$, there exists $\delta>0$
with the following property:
the polynomial $p_\eta$ has a root $z_\eta$ falling within an $\varepsilon$-neighborhood
of $z_0$, provided that $|\eta|<\delta$.
Choosing $\varepsilon<\Re z_0$, we are done.
\end{proof}

\begin{remark}
The request that $A$ has at least one eigenvalue is not really necessary to prove the result,
although it greatly simplifies the analysis.
If $A$ has no eigenvalues,
the same conclusion can be drawn by analyzing the spectrum of the generator
of the semigroup $S(t)$, which leads to the study of the equation $p_\eta(z)=0$,
for values of $\lambda$ in the spectrum of $A$.
In which case, the existence of a root with positive real part
provides a lower bound for the growth rate of the semigroup
(see \cite{DPMGT}).
\end{remark}

\section{Exponential Stability}
\label{EDEC}

\noindent
We now come to the core of our work,
regarding the uniform exponential decay of the solutions to~\eqref{main_sysA}.
Let us first recall two definitions.

\begin{definition}
System~\eqref{main_sysA}, or more precisely its related semigroup $S(t)$,
is said to be \emph{exponentially stable} if
\begin{equation}
\label{ExpDecayEstimate}
\E(t)\leq M\E(0)e^{-\omega t},
\end{equation}
for some constants $\omega>0$ and $M\geq 1$, both independent of the initial data
of the problem.
\end{definition}

\begin{definition}
\label{defexpoptimal}
The \emph{exponential decay rate} of the energy $\E$ is the best (in the sense of largest) $\omega>0$ for which exponential stability
holds; namely, it is the number
$$\omega_\star =\sup\big\{\omega>0:\,\,\E(t)\leq M\E(0)e^{-\omega t}\,\text{ for some }M\geq 1\big\}.$$
\end{definition}

\smallskip
In order to state a quantitatively precise result,
we define the \emph{stability threshold} $\t$ depending on
$\kappa$, besides the other structural quantities of the problem
but $\eta$,
as
\begin{equation}
\label{tikappa}
\t(\kappa)=
\begin{dcases*}
\frac{2\kappa}{\alpha^2}\bigg(1-\frac{\alpha}{\kappa\lambda_1} +\frac{\alpha^2}{\kappa^2\lambda_1^2}\bigg)
&if $\kappa < \alpha/\lambda_1$, \\
\frac{2\kappa}{\alpha^2} &if $\kappa \geq \alpha/\lambda_1$,
\end{dcases*}
\end{equation}
which attains its global minimum at $\kappa = \alpha / \lambda_1$.
Recall that $\lambda_1>0$ is the minimum of the spectrum of $A$.

\soglia

\begin{theorem}
\label{ThmDECAY}
Let $\alpha,\beta,\gamma>0$ be fixed and such that
$\mu=\gamma-\alpha\beta\geq 0$.
Assume that
\begin{equation}
\label{main_condition}
\eta^2 > \t(\kappa)\mu.
\end{equation}
Then system~\eqref{main_sysA}
is exponentially stable.
\end{theorem}

Theorem~\ref{ThmDECAY} tells that in the critical regime $\mu=0$ we have
exponential stability for every $\eta \neq 0$.
On the other hand, in the supercritical regime $\mu>0$, the exponential decay of the energy
depends on the interplay between the coupling constant $\eta$
and the Fourier coefficient $\kappa$.
Observe also that at the global minimum we have
$$
\t (\alpha/\lambda_1) = \frac{2}{\alpha\lambda_1}.
$$
Thus the value $\kappa=\alpha/\lambda_1$ is ``optimal", being the one requiring the weakest
coupling to drive the system to stability. For such a $\kappa$ condition~\eqref{main_condition}
reads
$$
\eta^2 > \frac{2\mu}{\alpha\lambda_1}.
$$
The proof, carried out in the next section, requires several steps.

\section{Proof of Theorem~\ref{ThmDECAY}}

\noindent
As customary, we will work with sufficiently regular solutions,
which stand the forthcoming calculations. The final conclusions follow by density.

\subsection{Auxiliary functionals}
We introduce the auxiliary functionals
\begin{align*}
\mathsf F(t) &= \eta \l \theta(t), u_t(t) \r + \frac{\eta^2}{2} \| u_t(t) \|^2_1 + \frac{1}{2} \|\theta(t)\|^2_{-1},\\
\noalign{\vskip1mm}
\mathsf G(t) &= -\l u_t(t) - \alpha u(t), u_{tt}(t) + \alpha u_t(t) \r.
\end{align*}

\begin{lemma}
\label{ineqF}
The following inequality holds:
$$
\frac{d}{dt} \mathsf F + \frac{\alpha \eta^2}{2} \|u_t(t)\|^2_1 \leq
\frac{ \kappa^2}{\alpha}\| \theta(t) \|^2_1 + \Big(\frac{\alpha}{\lambda_1} - \kappa \Big) \|\theta(t) \|^2.
$$
\end{lemma}

\begin{proof}
Via direct computations,
$$
\frac{d}{dt}\mathsf F +\alpha \eta^2 \|u_t\|^2_1= - \kappa \eta \l \theta, u_t \r_1
-  \kappa \|\theta \|^2 - \alpha \eta \l \theta, u_t \r.
$$
Using the Young and Poincar\'e inequalities, we estimate the terms in the right-hand side as
$$
-\alpha \eta \l \theta, u_t \r  \leq \frac{\alpha}{\lambda_1} \|\theta\|^2 + \frac{\alpha \eta^2}{4} \|u_t\|^2_1,
$$
and
$$
-\kappa \eta \l \theta, u_t \r_1 \leq \frac{ \kappa^2}{\alpha} \|\theta\|^2_1 + \frac{\alpha \eta^2}{4}\|u_t\|^2_1.
$$
This will do.
\end{proof}

\begin{remark}
In the Young inequalities above we used the weights $(1/2,1/2)$. One might argue that the more general
weights $(\nu/2,1/2\nu)$ should be used instead, looking eventually for $\nu>0$
producing the best estimate. However, doing so, one would realize \emph{a posteriori} that the optimal
value is indeed $\nu=1/2$.
\end{remark}

\begin{lemma}
\label{ineqPsi}
The following inequality holds:
$$
\frac{d}{dt} \mathsf G + \frac{\gamma}{2\alpha} \|u_t + \alpha u\|^2_1
+ \frac{1}{2} \|u_{tt} + \alpha u_t\|^2 \leq  \ell \|u_t\|^2_1+\eta \l \theta, u_t + \alpha u \r_1
- 2\eta \l \theta, u_t \r_1,
$$
where
$
\ell= \frac{4\gamma^2+\mu^2}{2\alpha\gamma} + \frac{2\alpha^2}{\lambda_1}.
$
\end{lemma}

\begin{proof}
Via direct computations,
\begin{align*}
&\frac{d}{dt} \mathsf G + \frac{\gamma}{\alpha} \|u_t + \alpha u\|^2_1
+ \|u_{tt} + \alpha u_t\|^2 \\
\noalign{\vskip1mm}
&=\frac{2\gamma+\mu}{\alpha} \l u_t, u_t + \alpha u\r_1
+ \eta \l \theta, u_t + \alpha u \r_1 \\
&\quad-2 \eta \l \theta, u_t \r_1 - \frac{2\mu}{\alpha} \|u_t\|^2_1 + 2\alpha \l u_t, u_{tt} + \alpha u_t \r.
\end{align*}
By the Young inequality,
\begin{align*}
\frac{2\gamma+\mu}{\alpha} \l u_t, u_t + \alpha u\r_1 &\leq \frac{\gamma}{2\alpha}\|u_t + \alpha u\|^2_1
+\frac{(2\gamma+\mu)^2}{2\alpha\gamma}\|u_t\|^2_1, \\
\noalign{\vskip1mm}
2\alpha \l u_t, u_{tt} + \alpha u_t \r &\leq \frac{1}{2}\|u_{tt} + \alpha u_t\|^2 + \frac{2\alpha^2}{\lambda_1}\|u_t\|^2_1,
\end{align*}
yielding the claim.
\end{proof}

We are now ready to
define our energy functional
$$
\mathsf L(t) = \mathsf W(t) + \rho \mathsf F(t) + \varepsilon^2 \mathsf G(t),
$$
for some $\rho>0$ and $\varepsilon>0$ to be chosen later.
Combining Lemma \ref{ineqF} and Lemma \ref{ineqPsi} together with
\eqref{EnEq}, we get
\begin{align*}
&\frac{d}{dt} \mathsf L + \frac{\varepsilon^2 \gamma}{2\alpha} \|u_t + \alpha u\|^2_1
+ \frac{\varepsilon^2}{2} \|u_{tt} + \alpha u_t \|^2 \\
\noalign{\vskip1mm}
&\quad +\Big( \frac{\rho \alpha \eta^2}{2} - 2 \mu - \varepsilon^2\ell \Big) \|u_t\|^2_1
+ \rho \Big(\kappa - \frac{\alpha}{\lambda_1}\Big) \|\theta \|^2 + \Big(2 \kappa
- \frac{\rho \kappa^2}{\alpha} \Big)\|\theta\|^2_1 \\
\noalign{\vskip1mm}
&\leq \varepsilon^2 \eta\l \theta,  u_t + \alpha u \r_1 - 2 \varepsilon^2  \eta\l \theta, u_t \r_1.
\end{align*}
\noindent
We control the right-hand side as
\begin{align*}
\varepsilon^2 \eta\l \theta,  u_t + \alpha u \r_1 &\leq \frac{\varepsilon \eta}{2} \|\theta\|^2_1
+ \frac{\varepsilon^3\eta}{2} \|u_t+ \alpha u\|^2_1, \\
\noalign{\vskip1mm}
 -2 \varepsilon^2 \eta\l \theta, u_t \r_1 &\leq \varepsilon \eta \|\theta\|^2_1 + \varepsilon^3\eta\|u_t\|^2_1.
\end{align*}
Hence, we end up with
\begin{align}
\label{exp_decaymah}
&\frac{d}{dt} \mathsf L + \frac{\varepsilon^2}{2} \Big( \frac{\gamma}{\alpha} - \varepsilon \eta \Big) \|u_t + \alpha u\|^2_1
+ \frac{\varepsilon^2}{2} \|u_{tt} + \alpha u_t \|^2 \\
&\quad+\Big( \frac{\rho \alpha \eta^2}{2} - 2 \mu - \varepsilon^2 \ell -\varepsilon^3 \eta \Big) \|u_t\|^2_1 \notag\\
&\quad+ \rho \Big(\kappa - \frac{\alpha}{\lambda_1}\Big)  \|\theta \|^2
+  \Big(2 \kappa - \frac{\rho \kappa^2}{\alpha} - \frac{3\varepsilon\eta}{2} \Big) \|\theta\|^2_1\leq0.\notag
\end{align}
At this point, we show that we can choose $\rho$ and $\varepsilon$ in such a way to obtain a satisfactory
differential inequality. Here, assumption~\eqref{main_condition} plays a crucial role.

\begin{lemma}
\label{ineqL}
Let \eqref{main_condition} hold. Then, there exists
$\rho>0$ with the following property: for every $\varepsilon>0$ sufficiently small,
there is $\omega=\omega(\varepsilon)>0$
such that
$$
\frac{d}{dt} \mathsf L + \omega \E\leq 0.
$$
\end{lemma}

\begin{proof}
Recalling the definition of $\t(\kappa)$ given in \eqref{tikappa},
we fix the value of $\rho>0$ to be
\begin{equation}
\label{epsss}
\rho= \frac{4(\mu+\kappa)}{\alpha \eta^2 +\alpha\kappa\t(\kappa)},
\end{equation}
and we set
$$
\sigma= \frac{2\kappa(\eta^2 - \t(\kappa)\mu)}{\eta^2+\kappa\t(\kappa)}.
$$
Observe that \eqref{main_condition} ensures that $\sigma> 0$.
It is convenient to consider two cases.

\medskip
\noindent
${\bf (i)}$  If $\kappa \leq \alpha / \lambda_1$
the coefficient
of $\|\theta\|^2$ in \eqref{exp_decaymah} is negative. So, we apply the Poincar\'e inequality to get
$$
\rho \Big(\kappa - \frac{\alpha}{\lambda_1}\Big)  \|\theta \|^2
+\Big(2 \kappa - \frac{\rho \kappa^2}{\alpha} - \frac{3\varepsilon\eta}{2} \Big) \|\theta\|^2_1
\geq \Big(2 \kappa - \frac{\rho \kappa^2}{\alpha} + \frac{\rho \kappa}{\lambda_1}
- \frac{\rho\alpha}{\lambda_1^2}  - \frac{3\varepsilon\eta}{2} \Big)\|\theta\|_1^2.
$$
By the choice of $\rho$ and the definition of $\sigma$, we see that
$$\frac{\rho\alpha \eta^2}{2} - 2 \mu = 2 \kappa - \frac{\rho \kappa^2}{\alpha} + \frac{\rho \kappa}{\lambda_1}
- \frac{\rho\alpha}{\lambda_1^2}=\sigma.$$
Accordingly, \eqref{exp_decaymah} becomes
\begin{align}
\label{ELLEmezzo}
&\frac{d}{dt} \mathsf L + \frac{\varepsilon^2}{2} \Big( \frac{\gamma}{\alpha} - \varepsilon \eta \Big) \|u_t + \alpha u\|^2_1
+ \frac{\varepsilon^2}{2} \|u_{tt} + \alpha u_t \|^2 \\
&\quad+(\sigma - \varepsilon^2 \ell -\varepsilon^3 \eta) \|u_t\|^2_1
+\frac{1}2(2\sigma - 3\varepsilon\eta) \|\theta\|^2_1\leq0.\notag
\end{align}
It is then apparent that when $\varepsilon$ is sufficiently small, all the coefficients
become positive.
Hence, calling
\begin{equation}
\label{omega_eta}
\omega = \min\Big\{\varepsilon^2\Big(\frac{\gamma}{\alpha} - \varepsilon \eta \Big),\,
\varepsilon^2,\, 2(\sigma - \varepsilon^2 \ell -\varepsilon^3 \eta),\, \lambda_1(2\sigma - 3\varepsilon\eta)\Big\},
\end{equation}
and making a further use of the Poincar\'e inequality, we arrive at the desired conclusion.

\medskip
\noindent
${\bf (ii)}$ If $\kappa > \alpha / \lambda_1$ the coefficient
of $\|\theta\|^2$ in \eqref{exp_decaymah} is positive and can be neglected. With $\rho$ and $\sigma$ as above, we
now have
$$\frac{\rho\alpha \eta^2}{2} - 2 \mu = 2 \kappa - \frac{\rho \kappa^2}{\alpha}=\sigma.$$
So we end up with \eqref{ELLEmezzo}, and we conclude
exactly as in the previous step.
\end{proof}

\subsection{Equivalence of the energy}
The next step is showing that the functional $\mathsf L$ is actually equivalent to the energy $\E$.

\begin{lemma}
\label{LemmaEQ}
Assume that
\begin{equation}
\label{NOR}
\eta^2>\frac{\mu}{\alpha\lambda_1}\qquad\text{and}\qquad
\rho>\frac{2\mu\lambda_1}{\alpha\eta^2\lambda_1-\mu}.
\end{equation}
Then, there exists $c>1$ such that
$$
\frac{1}{c}\, \E \leq \mathsf L \leq c\, \E,
$$
for every $\varepsilon>0$ small.
\end{lemma}

\begin{proof}
Throughout this proof, $c>1$ will denote a generic constant, depending only on $\rho$ and
the structural parameters
of the problem.
The inequality
$$\mathsf L \leq c\, \E$$
is a straightforward consequence of the Young and Poincar\'e inequalities
(recall that $\varepsilon$ is small).
By the same token, it is also clear that
$$
|\mathsf G|\leq c \,\E.
$$
Therefore, we are left to prove the remaining inequality
$$
\frac{1}{c}\, \E \leq \mathsf W + \rho \mathsf F.
$$
Since
$$
\mathsf W + \rho\mathsf F =
\|u_{tt} + \alpha u_t\|^2 + \frac{\gamma}{\alpha} \|u_t + \alpha u\|^2_1
+ \Big(\frac{\rho \eta^2}{2}-\frac{\mu}{\alpha}\Big) \| u_t \|^2_1
+ \rho \eta \l \theta, u_t \r + \|\theta\|^2  + \frac{\rho}{2} \|\theta\|^2_{-1},
$$
we only need to show that
$$
\Big(\frac{\rho \eta^2}{2}-\frac{\mu}{\alpha}\Big) \| u_t \|^2_1
+ \rho \eta \l \theta, u_t \r + \|\theta\|^2  + \frac{\rho}{2} \|\theta\|^2_{-1}
\geq c \Big[\|u_t\|^2_1+\|\theta\|^2\Big].
$$
Indeed,
$$\l \theta, u_t \r\geq - \frac{\eta \nu}{2} \|u_t\|^2_1 -\frac{1}{2 \eta \nu} \|\theta\|^2_{-1},
$$
for every $\nu\in (0,1)$.
Therefore,
\begin{align*}
&\Big(\frac{\rho \eta^2}{2}-\frac{\mu}{\alpha}\Big) \| u_t \|^2_1
+ \rho \eta \l \theta, u_t \r +\|\theta\|^2+ \frac{\rho}{2} \|\theta\|^2_{-1}\\
&\geq \Big(\frac{\rho\eta^2}{2}(1-\nu)-\frac{\mu}{\alpha}\Big) \| u_t \|^2_1
+\|\theta\|^2 - \frac{\rho}{2\nu}(1-\nu) \|\theta\|^2_{-1}\\
&\geq \Big(\frac{\rho\eta^2}{2}(1-\nu)-\frac{\mu}{\alpha}\Big) \| u_t \|^2_1 +
\Big(1 - \frac{\rho}{2\nu\lambda_1}(1-\nu) \Big)\|\theta\|^2.
\end{align*}
The conclusion follows if we find $\nu\in(0,1)$ for which the two coefficients in the right-hand side are positive,
which amounts to requiring, besides $\nu>0$,
$$
1-\frac{2\lambda_1}{2\lambda_1+\rho}<\nu<1-\frac{2\mu}{\rho\alpha\eta^2}.
$$
This is possible since
$$
\frac{\mu}{\rho\alpha\eta^2} <  \frac{\lambda_1}{2\lambda_1 + \rho}<\frac12,
$$
and the first inequality is nothing but \eqref{NOR}.
\end{proof}

\subsection{Conclusion of the proof}
We choose $\rho$ as in \eqref{epsss}, so that
Lemma~\ref{ineqL} applies.
Besides, we know from \eqref{main_condition} that
$$\rho > \frac{4\mu}{\alpha \eta^2},$$
and
$$\eta^2>\min_{\kappa>0}\,\t(\kappa)\mu=\frac{2\mu}{\lambda_1\alpha}.$$
Combining the two inequalities, we obtain
$$\rho(\alpha\eta^2\lambda_1-\mu)>4\mu\lambda_1-\frac{4\mu^2}{\alpha\eta^2}\geq 2\mu\lambda_1,$$
meaning that \eqref{NOR} holds true.
Therefore, Lemma~\ref{LemmaEQ} applies as well.
Accordingly, fixing $\varepsilon>0$ small enough,
and redefining $\omega$ up to the multiplicative constant $c$, we are led to
$$
\frac{d}{dt}\mathsf L + \omega \mathsf L \leq 0.
$$
The Gronwall lemma then gives
$$\mathsf L(t)\leq \mathsf L(0)e^{-\omega t}.$$
Since ${\mathsf L}\sim\E$, we finally arrive at the sought energy inequality~\eqref{ExpDecayEstimate}.
\qed

\section{Optimal Decay Rate}
\label{ODR}

\noindent
Once the exponential stability of the system is established,
another relevant issue concerns with the exponential decay rate $\omega_\star $ of the energy $\E$,
introduced in Definition~\ref{defexpoptimal}.
Indeed, one might want to look for the value of $\eta$ that
maximizes $\omega_\star $, once all the other parameters are fixed.
We already know from Section~\ref{EBU} that when $\eta\to 0$ exponential stability is lost.
Let us see what happens when $|\eta|\to \infty$, where exponential stability
certainly occurs in view of Theorem~\ref{ThmDECAY}.

\begin{proposition}
\label{propdet}
Let $\alpha,\beta,\gamma,\kappa$ be fixed. Then the exponential
decay rate  of the energy goes to zero as $|\eta|\to\infty$.
\end{proposition}

\begin{proof}
As in Proposition~\ref{PROPOPRIMA}, we assume that $A$ has at least one eigenvalue $\lambda>0$,
but again this simplifying assumption is not really necessary.
Looking for single-mode solutions to~\eqref{main_sysA},
we boil down once more to system~\eqref{ODE}. Accordingly,
all we need to show is that the characteristic polynomial $p_\eta$
of Section~\ref{EBU} has at least a root whose real part tends to zero
as $|\eta|\to\infty$. Note that
$$p_\eta(0)=\gamma \kappa \lambda^2>0.
$$
Besides, choosing
$$\varepsilon_\eta=\frac{\gamma\kappa\lambda^2+1}{\alpha\lambda^2}\,\frac1{\eta^2},$$
it is readily seen that
$$\lim_{|\eta|\to\infty}p_\eta(-\varepsilon_\eta)= -1.
$$
Thus, for $|\eta|$ large enough, $p_\eta(-\varepsilon_\eta)$ becomes negative, implying that
$p_\eta$ has a (negative) real root
$-x_\eta\in (-\varepsilon_\eta,0)$, where $\varepsilon_\eta\to 0$. This means that the exponential decay rate
is less than or equal to $2x_\eta\to 0$.
\end{proof}

\begin{remark}
We should observe that the conclusions of Proposition~\ref{propdet} hold no matter
if $\mu$ is positive or not. Thus, when $|\eta|\to\infty$ the decay rate
of the energy deteriorates also in the subcritical case $\mu<0$ studied in~\cite{BUR}.
\end{remark}

Accordingly, the exponential
decay rate, which is easily seen to be continuous in $\eta$, attains its maximum when $\eta=\pm\eta_\star $,
for some $\eta_\star>0$,
being clear that the picture depends only on the modulus of $\eta$.
This complies with our calculations:
when $\alpha,\beta,\gamma,\kappa$ are fixed, the best decay rate $\omega_b=\omega_b(\eta)$
predicted by Theorem~\ref{ThmDECAY} is obtained, up to a multiplicative constant,
by maximizing \eqref{omega_eta} with respect to $\varepsilon>0$.
In turn, the function $\eta\mapsto \omega_b(\eta)$ reaches its maximum for some $\eta=\pm\eta_b$,
with
$$\sqrt{\t(\kappa)\mu}\,<\eta_b<\infty,$$
although it is not easy at all to compute such an $\eta_b$ explicitly.
Nevertheless, this does not solve the problem of finding the actual value $\omega_\star $ of the
exponential decay rate (for any fixed $\eta$ large enough).
First, because our estimates provide only
sufficient conditions. But, more importantly, because it is generally impossible to
establish the exponential decay rate of a linear system via energy estimates. This fact will be discussed
in detail in the final Appendix.

Actually, there is a last interesting question to be addressed, namely, to see
what happens in the limit situations when the thermal conductivity $\kappa$ becomes
either very small or very large.
More precisely, for fixed $\alpha,\beta,\gamma$, let us consider the exponential decay rate
$$\omega_\star =\omega_\star (\kappa,\eta).$$
This function is defined for every $\kappa>0$ and every $\eta$ with $|\eta|$ large enough.
Then, for every $\kappa>0$, we set
$$\omega_\kappa=\max_{\eta}\,\omega_\star (\kappa,\eta).$$
In other words, $\omega_\kappa$ is the best possible decay rate for a fixed $\kappa$, which is
obtained by suitably modulating the coupling constant $\eta$.
It is reasonable to expect that $\omega_\kappa\to 0$ as $\kappa\to 0$.
This is indeed true, but not so impressive from the physical viewpoint (and we omit the proof). On the contrary,
it would seem reasonable to bet on $\omega_\kappa\to \infty$ as $\kappa\to \infty$.
The next proposition tells that you would lose.

\begin{proposition}
Let $\alpha,\beta,\gamma$ be fixed. Then there exists $\xi=\xi(\alpha,\beta,\gamma)>0$
such that
$$\limsup_{\kappa\to \infty}\,\omega_\kappa\leq 2\xi.$$
\end{proposition}

\begin{proof}
We parallel the proof of Proposition~\ref{propdet}, considering
the characteristic polynomial of Section~\ref{EBU},
that here we simply call $p$, which fulfills $p(0)=\gamma \kappa \lambda^2>0$.
We will establish the claim by showing that there exists $0<\zeta\leq \xi$ such that
$$p(-\zeta)<0,$$
for every $\kappa$ large enough. This means that $p$ has a  real root
in $(-\xi,0)$, implying that $\omega_\kappa\leq 2\xi$.
To this aim, we introduce the positive number (as $\mu\geq 0$)
$$
r = \frac{\alpha^2 \lambda}{\alpha^3 - 4\alpha\beta\lambda + 8\gamma\lambda},
$$
and we call $\xi$ the unique real solution to the equation
$$
-\xi^3 + \Big(\alpha + \frac{\lambda}{r} \Big)\xi^2 +\gamma\lambda= 0.
$$
It is straightforward to check that $\xi  > \alpha$.
Then, we split $p$ into the sum
$$
p(z) = f(z) + g(z),
$$
where
\begin{align*}
f(z)&=\kappa \lambda z^3 + (\alpha \kappa \lambda +
 \eta^2 \lambda^2) z^2 + (\beta \kappa \lambda^2 + \alpha \eta^2 \lambda^2 ) z + \gamma \kappa \lambda^2, \\
g(z)&= z^4 + \alpha z^3 + \beta \lambda z^2 + \gamma \lambda z.
\end{align*}
Note that $g$ is independent of $\kappa$ and $\eta$.
If $\kappa/\eta^2 \leq r$, we set $\zeta=\alpha/2$ to obtain
$$
f(-\zeta) =\frac{\alpha^2\lambda^2\eta^2}{8} \Big(-2+ \frac{\kappa}{\eta^2r} \Big)
\leq - \frac{\alpha^2\lambda^2\eta^2}{8}\leq - \frac{\alpha^2\lambda^2\kappa}{8r}.
$$
Instead, if $\kappa/\eta^2 > r$, we set $\zeta=\xi$ and, from the
very definition of $\xi$, we get
$$
f(-\zeta ) = -\kappa\lambda^2 \Big(\beta \xi +\frac{\kappa-r\eta^2}{r\kappa}\xi ^2
+\frac{\alpha \eta^2}{\kappa} \xi  \Big) \leq -\beta \xi \lambda^2 \kappa.
$$
In both cases, we conclude that
$p(-\zeta) =f(-\zeta) + g(-\zeta)\to -\infty$
as $\kappa\to\infty$.
\end{proof}

\begin{remark}
Actually, the same is true also
in the subcritical regime $\mu <0$.
Just note that
$p(-\alpha) = \lambda \mu (\kappa \lambda - \alpha)\to -\infty$
as $\kappa\to\infty$.
\end{remark}

\section{Comparison with Numerical Results}

\noindent
The aim of this section is to verify to what extent our theoretical findings
are consistent with the numerical simulations.
This is of some importance, since our main Theorem~\ref{ThmDECAY} establishes only sufficient conditions
on the exponential decay.
As a model, we consider the one-dimensional version of~\eqref{main_sysA}, i.e., where the underlying
Hilbert space $H$ is simply $\R$ and $A=1$, implying in turn $\lambda_1=1$.
As far as the MGT parameters are concerned, we take
$$\alpha=2,\quad \beta=1,\quad \gamma=3,
$$
meaning that we sit in the supercritical regime with $\mu=1$. But, in fact, nothing really changes
in the forthcoming analysis by selecting different parameters complying with $\mu>0$.
Accordingly, our system reads
$$
\begin{cases}
u''' + 2 u'' + u' + 3 u =\eta \theta, \\
\noalign{\vskip1mm}
\theta' + \kappa \theta =- \eta u'' - 2 \eta u',
\end{cases}
$$
where the \emph{prime} stands time-derivative,
and whose characteristic polynomial is
$$
p(z)= z^4 + (2 + \kappa)z^3 + (1 + 2\kappa +\eta^2)z^2
+(3+\kappa + 2\eta^2 ) z + 3\kappa .
$$

\subsection{The stability threshold} The first task is comparing the theoretical stability
threshold $\t(\kappa)$ provided by~\eqref{tikappa} with the actual stability threshold
$\t_\star (\kappa)$, which certainly exists in the light of Proposition~\ref{PROPOPRIMA}.
With the aid of Mathematica{\tiny\texttrademark}
we compute $\t_\star (\kappa)$ as
$$
\t_\star (\kappa)=\sup\big\{\eta^2:\,p\text{ has a root with positive real part}\big\},
$$
and we compare the two thresholds in the following figure.

\bigskip
\soglianumerica

\noindent
As expected, $\t(\kappa) > \t_\star (\kappa)$ for every $\kappa$. Nonetheless,
both functions have a linear growth for $\kappa$ large,
and they also exhibit the same behavior when $\kappa\to 0$.
This is quite clear from the graph of the ratio $\t/\t_\star$.

\bigskip
\ratio

\bigskip
\bigskip
\begin{remark}
Note that $\t(\kappa)$ does not really depend on the particular operator $A$, but only on
the minimum $\lambda_1$ of its spectrum.
However, we might add that, in the one-dimensional case, our theoretical results
are in fact better and closer to the numerical simulations, since in the estimates (precisely those of
Lemma~\ref{ineqL}) we can lean on the fact that the 0-norm and the 1-norm coincide.
Indeed, we actually have
$$\t(\kappa)=\frac{2\kappa}{\alpha^2}\bigg(1-\frac{\alpha}{\kappa\lambda_1}
+\frac{\alpha^2}{\kappa^2\lambda_1^2}\bigg),\quad\forall\kappa>0.$$
\end{remark}

\subsection{Optimal decay}
As already mentioned, the value $\omega_b$ found in Theorem~\ref{ThmDECAY},
is certainly worse than the actual decay rate $\omega_\star $,
the latter being easily computed from the roots of $p$.
Nonetheless, their shapes as functions of $\eta$ are similar,
and the maximum points are pretty close.
We perform the numerical analysis for $\kappa=2$,
and we recall that, up to a multiplicative constant,
$\omega_b$ is computed by means of~\eqref{omega_eta},
which for this particular choice of the parameters reads
$$
\omega =\omega(\eta)= \min\Big\{\varepsilon^2\Big(\frac32 - \varepsilon \eta \Big),\,
\varepsilon^2,\, 2\sigma - \frac{133}{6}\varepsilon^2
 - 2\varepsilon^3 \eta,\, 2\sigma - 3\varepsilon\eta\Big\}.
$$
Accordingly,
$$\omega_b=\omega_b(\eta)=\sup_{\varepsilon>0}\omega(\eta).$$

\bigskip
\doppio

\noindent
Note that $\eta_b \approx 1.86$ whereas
$\eta_\star\approx 2.15$.

\section*{Appendix:\\ Exponential Decay Rate and Energy Estimates}
\theoremstyle{definition}
\newtheorem{remarkAPP}{Remark}[section]
\renewcommand{\theremarkAPP}{A.\arabic{remarkAPP}}
\setcounter{equation}{0}
\setcounter{subsection}{0}
\renewcommand{\theequation}{A.\arabic{equation}}

\noindent
In general, for a given linear differential system, it is impossible
to attain the exponential decay rate (in the sense of Definition~\ref{defexpoptimal})
via energy estimates,
except in the trivial situation
where the energy inequality~\eqref{ExpDecayEstimate} holds for $M=1$.
To show this fact, we look at the simplest possible example, namely,
the semigroup generated by the second-order differential equation
\begin{equation}
\label{oscillator}
x''+ x'+ x= 0,
\end{equation}
modeling the position $x = x(t) \in \R$ of an oscillator
subject to dynamical friction.
Setting $y=x'$, the total energy of \eqref{oscillator}, which is the sum of
the elastic potential and the kinetic energies, reads
$$
\mathsf e = \frac12 \big[x^2 +y^2\big].
$$
The exponential decay rate of $\mathsf e$
is completely determined by the complex roots of the characteristic polynomial
$$
p(\lambda) = \lambda^2 + \lambda + 1.
$$
Since the (conjugate) roots have both real part equal to
$-1/2$, we deduce the exponential decay rate
$\omega_\star =1$.
Let us see now what happens if we try to attain the exponential decay
via energy estimates.
For the case under investigation, one multiplies equation~\eqref{oscillator}
by $y + \varepsilon x$, for
$\varepsilon \in (0,1)$. Incidentally, such a multiplier
is the classical one used in the asymptotic analysis
of hyperbolic ODEs and PDEs.
Then, straightforward computations entail the differential identity
\begin{equation}
\label{eq_oscillator}
\frac{d}{dt} \mathsf g + \mathsf f=0,
\end{equation}
having set
$$\mathsf g=(1+\varepsilon)x^2 + y^2 + 2\varepsilon x y
\qquad\text{and}\qquad \mathsf f= 2\varepsilon x^2 + 2(1-\varepsilon) y^2.
$$
Both $\mathsf g$
and $\mathsf f$ are equivalent to the energy $\mathsf e$, due to the fact that
$\varepsilon \in (0,1)$.
At this point, the (only possible) strategy is showing that
\begin{equation}
\label{ineq_fL}
\omega \mathsf g\leq \mathsf f,
\end{equation}
for some $\omega >0$. In which case, one infers from \eqref{eq_oscillator} the differential inequality
$$\frac{d}{dt} \mathsf g +\omega\mathsf g\leq 0,$$
and an application of the Gronwall lemma yields
$$\mathsf g(t)\leq \mathsf g(0)e^{-\omega t}.$$ From the equivalence
$\mathsf g\sim\mathsf e$,
we conclude that $\mathsf e$ decays at an exponential rate $\omega$ as well.
In fact, we are looking for the best possible $\omega$ that can be obtained by means of this procedure,
which coincides with the largest  $\omega$ such that \eqref{ineq_fL} holds or, equivalently,
with the largest $\omega$ such that the inequality
$$
(2\varepsilon - \omega-\omega\varepsilon ) t^2 - 2 \omega \varepsilon t + 2-2\varepsilon - \omega \geq 0
$$
is verified for all $t\in\R$.
This translates into the conditions
\begin{equation} 	
\label{const}
\begin{dcases}
2\varepsilon - \omega-\omega\varepsilon > 0, \\
(2 \omega \varepsilon)^2 - 4 (2\varepsilon - \omega-\omega\varepsilon )(2-2\varepsilon - \omega) \leq 0.
\end{dcases}
\end{equation}
We have fallen into a simple optimization problem in two variables:
maximize the function
$$
F(\varepsilon, \omega) = \omega
$$
on the admissible region
$$
\mathfrak A = \big \{(\varepsilon, \omega) \in (0,1)\times(0,1]\,\text{ such that \eqref{const} holds}\big \}.
$$
The maximum of the (linear) extension of $F$ on the closure of $\mathfrak A$
is attained on the boundary. Since equality in the first constraint of \eqref{const} leads to a minimum,
$F$ assumes its maximum when equality holds
in the second constraint, yielding
$$
\omega = 1 - \sqrt{\frac{1-3\varepsilon+3\varepsilon^2}{1+\varepsilon-\varepsilon^2}}.
$$
By elementary computations, the latter quantity is maximized when
$\varepsilon = 1/2$, where $\omega =\omega_b$ with
$$\omega_{b}= 1 - \frac{\sqrt 5}{5}.$$
Since the pair $(1/2,\omega_b)\in\mathfrak A$, the problem is solved.
Summarizing, the best possible decay obtained through this method
is exactly
$\omega_b\approx 0.55$, with a sizable gap compared to the
actual decay rate $\omega_\star =1$.

\begin{remarkAPP}
As we mentioned above, for a linear semigroup $S(t)$ acting on a Hilbert space $\H$,
the exponential decay rate is attained
via energy estimates when~\eqref{ExpDecayEstimate} is satisfied with $M=1$.
This occurs if and only the inequality
$$\l {\mathbb A} u,u\r_\H\leq -\frac{\omega}2\|u\|_\H^2$$
holds true for all vectors $u$ in the domain of ${\mathbb A}$,
being ${\mathbb A}$ the infinitesimal generator
of $S(t)$.
Certainly this is not the case for the system~\eqref{main_sysA} of this paper, and not even for
the much simpler equation~\eqref{oscillator}.
It is however true for the semigroup generated by the linear parabolic equation
$$u_t+A u=0,$$
with $A$ as in Section~\ref{SFSN}.
\end{remarkAPP}

\subsection*{Acknowledgments}
We thank the anonymous Referee for useful suggestions and comments.


\bigskip


\bigskip
\begin{thebibliography}{00}

\bibitem{R1}
{\au F. Alabau-Boussouira, Z. Wang, L. Yu},
{\ti A one-step optimal energy decay formula for indirectly nonlinearly damped hyperbolic systems coupled by velocities},
{\jou ESAIM Control Optim.\ Calc.\ Var.}
\no{23}{721--749}{2017}

\bibitem{BUR}
{\au M.S. Alves, C. Buriol, M.V. Ferreira, J. E. Mu\~noz Rivera, M. Sep\'ulveda, O. Vera},
{\ti Asymptotic behaviour for the vibrations modeled by the standard linear
solid model with a thermal effect},
{\jou J.\ Math.\ Anal.\ Appl.}
\no{399}{472--479}{2013}

\bibitem{R2}
{\au K. Ammari, S. Nicaise},
{\ti Stabilization of a transmission wave/plate equation},
{\jou J.\ Differential Equations}
\no{249}{707--727}{2010}

\bibitem{R3}
{\au G. Avalos, I. Lasiecka},
{\ti Exponential stability of a thermoelastic system without mechanical dissipation},
{\jou Rend.\ Istit.\ Mat.\ Univ.\ Trieste}
\no{28}{1--28}{1997}

\bibitem{R4}
{\au G. Avalos, I. Lasiecka},
{\ti The strong stability of a semigroup arising from a coupled hyperbolic/parabolic system},
{\jou Semigroup Forum}
\no{57}{278--292}{1998}

\bibitem{R5}
{\au G. Avalos, I. Lasiecka, R. Triggiani},
{\ti Heat-wave interaction in 2-3 dimensions: optimal rational decay rate},
{\jou J.\ Math.\ Anal.\ Appl.}
\no{437}{782--815}{2016}

\bibitem{BU1}
{\au F. Bucci, I. Lasiecka},
{\ti Feedback control of the acoustic pressure in ultrasonic wave propagation},
{\jou Optimization}
\no{68}{1811--1854}{2019}

\bibitem{BU2}
{\au F. Bucci, L. Pandolfi},
{\ti On the regularity of solutions to the Moore-Gibson-Thompson equation: a perspective via wave equations with memory},
{\jou  J.\ Evol.\ Equ.}
\no{20}{837--867}{2020}

\bibitem{CLP}
{\au M. Conti, L. Liverani, V. Pata},
{\ti A note on the energy transfer in coupled differential systems},
{\jou Commun.\ Pure Appl.\ Anal.}, to appear.

\bibitem{CPPQ}
{\au M. Conti, V. Pata, M. Pellicer, R. Quintanilla},
{\ti On the analyticity of the MGT-viscoelastic plate with heat conduction},
{\jou J.\ Differential Equations}
\no{269}{7862--7880}{2020}

\bibitem{CPPQ2}
{\au M. Conti, V. Pata, M. Pellicer, R. Quintanilla},
{\ti A new approach to MGT-thermoviscoelasticity},
{\jou Discrete Contin.\ Dyn.\ Syst.},
to appear.

\bibitem{R6}
{\au M. Conti, V. Pata, R. Quintanilla},
{\ti Thermoelasticity of Moore-Gibson-Thompson type with history dependence in the temperature},
{\jou Asymptot.\ Anal.}
\no{120}{1--21}{2020}

\bibitem{CZW}
{\au R.F. Curtain, H. Zwart},
{\bk An introduction to infinite-dimensional linear systems theory},
\eds{Springer-Verlag}{New York}{1995}

\bibitem{DAC}
{\au B. D'Acunto, A. D'Anna, P. Renno},
{\ti On the motion of a viscoelastic solid in presence of a rigid wall},
{\jou Z.\ Angew.\ Math.\ Phys.}
\no{34}{421--438}{1983}

\bibitem{DLP}
{\au F. Dell'Oro, I. Lasiecka, V. Pata},
{\ti The Moore-Gibson-Thompson equation with memory in the critical case},
{\jou J.\ Differential Equations}
\no{261}{4188--4222}{2016}

\bibitem{R7}
{\au F. Dell'Oro, V. Pata},
{\ti On the stability of Timoshenko systems with Gurtin-Pipkin thermal law},
{\jou J.\ Differential Equations}
\no{257}{523--548}{2014}

\bibitem{DPMGT}
{\au F. Dell'Oro, V. Pata},
{\ti On the Moore-Gibson-Thompson equation and its relation to linear viscoelasticity},
{\jou  Appl.\ Math.\ Optim.}
\no{76}{641--655}{2017}

\bibitem{R8}
{\au R. Denk, F. Kammerlander},
{\ti Exponential stability for a coupled system of damped-undamped plate equations},
{\jou IMA\ J.\ Appl.\ Math.}
\no{83}{302--322}{2018}

\bibitem{DRO}
{\au A.D. Drozdov, V.B. Kolmanovskii},
{\bk Stability in viscoelasticity},
\eds{North-Holland}{Amsterdam}{1994}

\bibitem{R9}
{\au H.D. Fern\'andez Sare, R. Racke},
{\ti On the stability of damped Timoshenko systems: Cattaneo versus Fourier law},
{\jou Arch. Rational Mech. Anal.}
\no{194}{221--251}{2009}

\bibitem{GOB}
{\au G.C. Gorain, S.K. Bose},
{\ti Exact controllability and boundary stabilization of torsional vibrations of an internally damped flexible space structure},
{\jou J.\ Optim.\ Theory Appl.}
\no{99}{423--442}{1998}

\bibitem{R10}
{\au J. Hao, Z. Liu},
{\ti Stability of an abstract system of coupled hyperbolic and parabolic equations},
{\jou Z.\ Angew.\ Math.\ Phys.}
\no{64}{1145--1159}{2013}

\bibitem{HS}
{\au M.W. Hirsch, S. Smale},
{\bk Differential equations, dynamical systems and linear algebra},
\eds{Academic Press}{New York}{1974}

\bibitem{KLM}
{\au B. Kaltenbacher, I. Lasiecka, R. Marchand},
{\ti Wellposedness and exponential decay rates for the Moore-Gibson-Thompson
equation arising in high intensity ultrasound},
{\jou Control Cybernet.}
\no{40}{971--988}{2011}

\bibitem{KLP}
{\au B. Kaltenbacher, I. Lasiecka, M.K. Pospieszalska},
{\ti Well-posedness and exponential decay of the energy in the nonlinear Jordan-Moore-Gibson-Thompson
equation arising in high intensity ultrasound},
{\jou Math.\ Models Methods Appl.\ Sci.}
\no{22}{n.11, 1250035}{2012}

\bibitem{KN}
{\au B. Kaltenbacher, V. Nikoli\'c},
{\ti The Jordan-Moore-Gibson-Thompson equation:
well-posedness with quadratic gradient nonlinearity and singular limit
for vanishing relaxation time},
{\jou Math.\ Models Methods Appl.\ Sci.}
\no{29}{2523--2556}{2019}

\bibitem{R11}
{\au J.U. Kim},
{\ti On the energy decay of a linear thermoelastic bar and plate},
{\jou SIAM\ J.\ Math.\ Anal.}
\no{23}{889--899}{1992}

\bibitem{R12}
{\au Z. Liu, S. Zheng},
{\ti Exponential stability of the Kirchhoff plate with thermal or viscoelastic damping},
{\jou Quarterly Appl.\ Math.}
\no{55}{551--564}{1997}

\bibitem{R13}
{\au Z. Liu, S. Zheng},
{\bk Semigroups associated with dissipative systems},
\eds{Chapman \& Hall/CRC}{Boca Raton}{1999}

\bibitem{R14}
{\au J.E. Mu{\~n}oz Rivera, R. Racke},
{\ti Smoothing properties, decay, and global existence of solutions to nonlinear coupled systems of thermoelastic type},
{\jou SIAM J.\ Math.\ Anal.}
\no{26}{1547--1563}{1995}

\bibitem{R15}
{\au J.E. Mu\~noz Rivera, R. Racke},
{\ti Large solutions and smoothing properties for nonlinear thermoelastic systems},
{\jou J. \ Differential Equations}
\no{127}{454--483}{1996}

\bibitem{R16}
{\au J.E. Mu\~noz Rivera, R. Racke},
{\ti Mildly dissipative nonlinear Timoshenko systems--global existence and exponential stability},
{\jou J. \ Math. \ Anal. \ Appl.}
\no{276}{248--278}{2002}

\bibitem{R17}
{\au J.E. Mu\~noz Rivera, R. Racke},
{\ti Transmission problems in (thermo)viscoelasticity with Kelvin-Voigt damping: nonexponential, strong, and polynomial stability},
{\jou SIAM J.\ Math.\ Anal.}
\no{49}{3741--3765}{2017}

\bibitem{MDT}
{\au R. Marchand, T. McDevitt, R. Triggiani},
{\ti An abstract semigroup approach to the third-order Moore-Gibson-Thompson partial
differential equation arising in high-intensity ultrasound:\ structural decomposition,
spectral analysis, exponential stability},
{\jou Math.\ Methods Appl.\ Sci.}
\no{35}{1896--1929}{2012}

\bibitem{PAZ}
{\au A. Pazy},
{\bk Semigroups of linear operators and applications to partial differential equations},
\eds{Springer-Verlag}{New York}{1983}

\bibitem{MOG}
{\au F.K. Moore, W.E. Gibson},
{\ti Propagation of weak disturbances
in a gas subject to relaxation effects},
{\jou J.\ Aero/Space Sci.}
\no{27}{117--127}{1960}

\bibitem{R18}
{\au M.L. Santos, D.S. Almeida J\'unior, J.E. Mu{\~n}oz Rivera},
{\ti The stability number of the Timoshenko system with second sound},
{\jou J.\ Differential Equations}
\no{253}{2715--2733}{2012}

\bibitem{STO}
{\au Professor Stokes},
{\ti An examination of the possible effect of the radiation of heat on the propagation of sound},
{\jou Philos.\ Mag. Series 4}
\no{1}{305--317}{1851}

\bibitem{THO}
{\au P.A. Thompson},
{\bk Compressible-fluid dynamics},
\eds{McGraw-Hill}{New York}{1972}

\bibitem{R19}
{\au R. Triggiani},
{\ti Heat-viscoelastic plate interaction: analyticity, spectral analysis, exponential decay},
{\jou Evol.\ Equ.\ Control Theory}
\no{7}{153--182}{2018}

\bibitem{US}
{\au  D.J. Uherka, A.M. Sergott},
{\ti On the continuous dependence of the roots of a polynomial on its coefficients}
{\jou Amer.\ Math.\ Monthly}
\no{84}{368--370}{1977}

\end{thebibliography}
\end{document}